\documentclass{amsart}
\usepackage[utf8]{inputenc}
\usepackage{mathrsfs,amsthm}
\usepackage{color}
\usepackage[all]{xy}

\newcommand{\CB}{{\mathrm{CB}}}
\newcommand{\mCB}{{\mathrm{MCB}}}
\renewcommand{\P}{{\mathbb{P}}}

\title{A Cayley--Bacharach theorem and plane configurations}
\date{January 2021}

\theoremstyle{plain}
\newtheorem{theorem}{Theorem}
\newtheorem{proposition}[theorem]{Proposition}
\newtheorem{corollary}[theorem]{Corollary}
\newtheorem{conjecture}[theorem]{Conjecture}
\newtheorem{lemma}[theorem]{Lemma}

\theoremstyle{definition}
\newtheorem{definition}[theorem]{Definition}
\newtheorem{example}[theorem]{Example}

\theoremstyle{remark}
\newtheorem{remark}[theorem]{Remark}

\newtheorem{question}[theorem]{Question}
\newtheorem{problem}[theorem]{Problem}

\numberwithin{figure}{section}
\numberwithin{theorem}{section}
\numberwithin{equation}{section}

\newcommand{\mc}{\mathcal}
\newcommand{\mb}{\mathbb}

\author{Jake Levinson}
\address{
	Mathematics Department \\
	Simon Fraser University \\
	Burnaby, BC V5A 1S6 \\
	Canada}
\thanks{The first author was partly supported by an AMS Simons Travel Grant.}
\email{jake\_levinson@sfu.ca}

\author{Brooke Ullery}
\address{
	Mathematics Department \\
	Emory University \\
	Atlanta, GA 30322 \\
	USA}
\email{bullery@emory.edu}



\commby{Claudia Polini}

\begin{document}

\begin{abstract}
In this paper, we examine linear conditions on finite sets of points in projective space implied by the Cayley--Bacharach condition. In particular, by bounding the number of points satisfying the Cayley--Bacharach condition, we force them to lie on unions of low-dimensional linear spaces.

These results are motivated by investigations into degrees of irrationality of complete intersections, which are controlled by minimum-degree rational maps to projective space. As an application of our main theorem, we describe the fibers of such maps for certain complete intersections of codimension two.
\end{abstract}

\maketitle

\section{Introduction} \label{sec:introduction}

Let $\Gamma \subset \P^n$ be a finite set of points. We say that $\Gamma$ \textbf{satisfies the Cayley--Bacharach condition with respect to degree $r$ polynomials} or simply \textbf{is $\CB(r)$} if, whenever a homogeneous degree $r$ polynomial $F$ vanishes at all but one point of $\Gamma$, $F$ vanishes at the last point.

This notion was familiarized by the classical 19th century theorem of Cayley and Bacharach that states that if two plane curves of degree $d$ and $e$, respectively, meet in $de$ points, then any curve of degree $d+e-3$ that contains all but one of the intersection points must contain them all. That is, this set of $de$ points is $\CB(d+e-3)$. Since then, there have been many generalizations and conjectures based on this result. See \cite{EGH} for a thorough history and several questions that remain open; see \cite{KLR19} for a more algebraic viewpoint, particularly for the case where $r+1$ is the Castelnuovo--Mumford regularity of $\Gamma$.

The Cayley--Bacharach condition has come up recently in the context of measures of irrationality for projective varieties. Specifically, if $X$ is a smooth variety of dimension $n$ and $\Gamma \subset X$ is a finite set of points, we  say $\Gamma$ satisfies the Cayley--Bacharach condition with respect to a linear system $|V|$ if any section $s \in V$ vanishing at all but one point of $\Gamma$ vanishes at the last point. In the case where $\Gamma$ is a general fiber of a generically finite dominant rational map $X \dashrightarrow \P^n$, Bastianelli showed that $\Gamma$ satisfies the Cayley--Bacharach condition with respect to the canonical linear system $|K_X|$ (see \cite[Proposition 4.2]{Bas}). 

Several recent papers have exploited the Bastianelli result to compute the gonality, or the degree of irrationality, a higher dimensional analog of gonality, of several classes of smooth complete intersection varieties in projective space (see e.g. \cite{BCD}, \cite{BDELU}, \cite{SU}, \cite{HLU}). The theme is that if the canonical bundle of $X$ is sufficiently positive (as, for example, in the case of high degree complete intersections), the fibers are forced to lie in special positions.

For instance, to this end, Bastianelli, Cortini, and De Poi prove the following: 

\begin{theorem}{\cite[Lemma 2.4]{BCD}}
If $\Gamma \subset \P^n$ is $\CB(r)$ and \begin{equation}\label{eqn}|\Gamma| \leq 2r+1,\end{equation} then $\Gamma$ lies on a line.
\end{theorem}

This result allows them to show that the fibers of rational maps to $\P^r$ from a high degree hypersurface $X \subset \P^{r+1}$ are collinear. 

In other (non-hypersurface) examples, however, we do not expect the fibers to be collinear. A natural question, then, is: what can one say about the geometry of $\Gamma$ when the bound (\ref{eqn}) is weakened? For instance, when the bound is raised to $|\Gamma| \leq \tfrac{5}{2}r +1$, Stapleton and the second author show in \cite{SU} that $\Gamma$ lies on either a plane conic or two skew lines (see also Remark \ref{SUrem}).

In this paper, we examine linear conditions on $\Gamma$ implied by the Cayley--Bacharach condition. We call a union $\mc{P} = P_1 \cup \cdots \cup P_k \subseteq \mathbb{P}^n$ of positive-dimensional linear spaces a {\bf plane configuration of dimension} $\dim(\mc{P}) = \sum \dim(P_i)$ {\bf and length} $\ell(\mc{P}) = k$. We propose the following conjecture:

\begin{conjecture} \label{conj:main}
Let $\Gamma \subset \mathbb{P}^n$ be a finite set of points satisfying $\CB(r)$. If 
\begin{equation} \label{eqn:conj-bound}
|\Gamma| \leq (d+1)r + 1,
\end{equation}
then $\Gamma$ lies on a plane configuration $\mc{P}$ of dimension $d$.
\end{conjecture}

Note that this bound is easily shown to be tight: a counterexample for $(d+1)r+2$ is given by points on a rational normal curve (see Example  \ref{ex:rational-normal-curve}).

Our main result is as follows.

\begin{theorem} \label{thm:main}
Conjecture \ref{conj:main} holds in the following cases:
\begin{itemize}
    \item[(i)] For $r \leq 2$ and all $d$. Moreover, we may take $\ell(\mc{P}) = 1$. 
    \item[(ii)] For all $r$ and for $d \leq 3$. Moreover, for $r \leq 4$, we may take $\ell(\mc{P}) \leq 2$. \item[(iii)] For $d=4$ and $r=3$. Moreover, we may take $\ell(\mc{P}) \leq 2$.
\end{itemize}
\end{theorem}

We believe a stronger statement is true in at least two respects. First, for each $d$, we believe the \emph{length} of the plane configuration can be bounded in terms of $r$, as in the above theorem. Note that this maximum length can become arbitrarily large as $r$ increases (Example \ref{ex:d-lines}). See also Corollary \ref{cor:d-lines-bound}, which gives constraints on $\frac{r}{d-1}$ in the case where $\mc{P}$ has maximum length, i.e. $\mc{P}$ is $d$ disjoint lines.

Second, we believe $\Gamma$ lies on certain unions of curves of low degree and arithmetic genus, whose exact description depends on refinements of the bound on $|\Gamma|$, as in the result of Stapleton and the second author. See Section \ref{sec:conclusion} for additional discussion.

Theorem \ref{thm:main} directly implies a result about the geometry of fibers of low degree maps to projective space from certain codimension two complete intersections:

\begin{theorem}\label{codim2thm}
Assume $n \geq 0$. Let $X \subset \P^{n+2}$ be an irreducible complete intersection of a quartic hypersurface Y and a hypersurface of degree $a\geq 4n-5.$
\begin{enumerate}
    \item[(a)] If $f: X \dashrightarrow \P^n$ is a finite rational map of degree $\leq 3a$ and $\Gamma \subset X$ is a fiber of $f$, then $\Gamma$ lies on a plane configuration of dimension $3$.
    
    \item[(b)] In particular, %if $n \geq 2$, or more generally if $Y$ contains a line, 
    the conclusion of (a) holds for any dominant rational map $f: X \dashrightarrow \P^n$ of minimum degree.
\end{enumerate}

\end{theorem}

\begin{remark}\label{SUrem}
The cases of hypersurfaces of quadrics and cubics were dealt with in \cite{SU}. The authors showed that in the former case, the fibers must be collinear, and in the latter, the fibers must lie on plane conics or two skew lines -- in particular, the fibers lie on plane configurations of dimension two. 
\end{remark}

Applying the conclusion of Conjecture \ref{conj:main}, we obtain the following more general conjectural result about codimension two complete intersections of arbitrary degree.

\begin{proposition}\label{codim2prop}
Assume $n \geq 0$. Let $X \subset \P^{n+2}$ be an irreducible complete intersection of hypersurfaces $Y_a$ and $Y_b$, of degrees $a, b \geq 2$, respectively. Suppose that Conjecture \ref{conj:main} holds for $d = b-1$ and $r = a+b-n-3$ and that $a+1 \geq b(3+n-b) .$
\begin{enumerate}
    \item[(a)] If $f: X \dashrightarrow \P^n$ is a finite rational map of degree $\leq a(b-1)$ and $\Gamma \subset X$ is a fiber of $f$, then $\Gamma$ lies on a plane configuration of dimension $b-1$.
    
    \item[(b)] In particular, the conclusion of (a) holds for any dominant rational map $f: X \dashrightarrow \P^n$ of minimum degree.
\end{enumerate}

\end{proposition}

The paper is organized as follows. In Section \ref{sec:preliminaries}, we introduce the notion of a plane configuration and prove basic results about the Cayley--Bacharach condition and plane configurations which we make frequent use of throughout the remainder of the paper. In Section \ref{sec:examples}, we give some concrete examples. Sections \ref{sec:key-lemmas} and \ref{sec:proof-main} are devoted to the proof of the main theorem and requisite lemmas. Section \ref{sec:codim-2-ci} gives a short proof of Theorem \ref{codim2thm} and Proposition \ref{codim2prop}. Finally, in Section \ref{sec:conclusion}, we discuss directions for potential research and leave the reader with several open questions.

The authors would like to thank Nic Ford, Joe Harris, Rob Lazarsfeld, Jessica Sidman, and David Stapleton.

\section{Preliminaries} \label{sec:preliminaries}

Our main object of interest is the following.

\begin{definition}
A {\bf plane configuration} $\mathcal{P}$ is a union of distinct positive-dimensional linear spaces $\mathcal{P} =  \bigcup_{i=1}^k P_i$ in $\mathbb{P}^n$. We say its {\bf dimension} is $\dim \mathcal{P} = \sum_i \dim P_i$ and its {\bf length} is $\ell(\mathcal{P}) = k$.
\end{definition}

For $\mathcal{P}$ or any collection of linear spaces $P_1, \ldots, P_k$, we write $\mathrm{span}(\mathcal{P})$ or $\mathrm{span}(P_i : i=1, \ldots, k)$ for the linear span of their union.

Thus a $2$-dimensional plane configuration is either two lines or one $2$-plane. In a sufficiently large $\mathbb{P}^n$, if the planes of $\mc{P}$ are in general position then $\mathrm{span}(\mc{P})$ has dimension $\dim \mc{P} + \ell(\mc{P}) - 1$. Note that, by definition, a $0$-dimensional plane configuration is empty.

\begin{definition}
We say a plane configuration $\mc{P}$ is {\bf skew} if $P_i \cap P_j = \emptyset$ for all $i \ne j$. We say it is {\bf split} if, in addition, $\mc{P}$ is a projectivized direct sum decomposition of the affine cone over $\mathrm{span}(\mc{P})$. Equivalently,
\[\dim \mathrm{span}(\mc{P}) = \dim(\mc{P}) + \ell(P) - 1.\]
Equivalently, for all $i$, $P_i \cap \mathrm{span}(P_j : j \ne i) = \emptyset$.
\end{definition}

If there exist $i \ne j$ for which $P_i \cap P_j$ is nonempty, it will frequently be preferable for us to replace $P_i$ and $P_j$ by $\mathrm{span}(P_i, P_j)$. The resulting plane configuration has lower length than $\mc{P}$ and at most the same dimension, while also being larger (as a variety).
Since we wish to cover finite sets $\Gamma \subseteq \mathbb{P}^n$ by low-dimensional plane configurations, we may often freely reduce to $\mc{P}$ being skew.

Note that when $\ell(\mc{P}) = 2$, $\mc{P}$ is skew if and only if $\mc{P}$ is split.

\begin{example}
In $\mathbb{P}^4$, let $\mathcal{P}$ consist of three general lines. Then $\dim \mathcal{P} = 3$ and $\ell(\mathcal{P}) = 3$, but $\mathrm{span}(\mc{P}) = \mb{P}^4$, so $\mc{P}$ is skew but not split.
\end{example}

\subsection{Basic facts}

We now prove some basic facts about plane configurations and the Cayley--Bacharach condition. We will use them throughout the remaining arguments of the paper, often without explicit reference.

\begin{proposition} \label{prop:extend-to-hyperplane}
Let $\Gamma \subseteq \mb{P}^n$ be a finite set and $P$ a linear space. Then $P$ can be extended to a hyperplane $H$ containing no additional points of $\Gamma$. That is, $\Gamma \cap H = \Gamma \cap P$.
\end{proposition}
\begin{proof}
Omitted.
\end{proof}

\begin{proposition}[Excision property] \label{prop:complements}
Let $\Gamma \subseteq \mb{P}^n$ be a finite set satisfying $\CB(r)$. Let $\mc{P}$ be a plane configuration of length $\ell$. Then $\Gamma \setminus \mc{P}$ is $\CB(r-\ell)$.
\end{proposition}

In particular, the complement of $\Gamma$ by a single linear space is $\CB(r-1)$.

\begin{proof}
It suffices to show that, for a single linear space $P$, the set $\tilde\Gamma := \Gamma \setminus P$ is $\CB(r-1)$. By extending $P$ and avoiding $\tilde \Gamma$, we may replace $P$ by a hyperplane $H$.

Let $Z$ be a hypersurface of degree $r-1$ containing all but one point $x \in \tilde \Gamma$. Then $Z \cup H$ has degree $r$ and contains $\Gamma \setminus \{x\}$. Since $\Gamma$ is $\CB(r)$, $Z \cup H$ contains the final point $x$. By construction $H$ does not contain $x$, so $Z$ does.
\end{proof}

\begin{proposition}[Basic lower bound]\label{prop:min-count}
Let $\Gamma \subseteq \mb{P}^n$ be a finite set satisfying $\CB(r)$. If $\Gamma$ is nonempty, $|\Gamma| \geq r+2$.
\end{proposition}

\begin{proof}
Suppose $1 \leq |\Gamma| \leq r+1$. Fix $x \in \Gamma$ and find hyperplanes $H_i$, each containing exactly one of the other points of $\Gamma$. The union of the $H_i$'s has degree $|\Gamma| - 1 \ \leq r$ and misses $x \in \Gamma$, a contradiction.
\end{proof}
In particular, Conjecture \ref{conj:main} holds in the case $d=0$ (recall that a $0$-dimensional plane configuration is empty).

We also record the following stronger lower bound, whose proof is trivial, for use in inductive arguments below.
\begin{proposition}[Inductive lower bound]
\label{prop:inductive-lower-bound}
Assume Conjecture \ref{conj:main} holds for fixed $r$ and dimension up to $d-1$. Let $\Gamma \subseteq \mathbb{P}^n$ such that $|\Gamma| \leq (d+1)r+1$, satisfying $\CB(r)$. Then if $\Gamma$ does \emph{not} lie on a plane configuration of dimension $d-1$, $|\Gamma| \geq dr+2$.
\end{proposition}

\section{Examples} \label{sec:examples}

Our conjecture is sharp in certain respects and not others, as the following examples demonstrate.

\begin{example}[Rational normal curves] \label{ex:rational-normal-curve}
Let $r \geq 2$ and $m = (d+1)r+2$. Let $\Gamma \subset \mathbb{P}^{d+1}$ be any set of $m$ distinct points on a rational normal curve $C \subset \mathbb{P}^{d+1}$ of degree $d+1$. Then $\Gamma$ is in general linear position, hence is \emph{not} contained in a plane configuration of dimension $d$. (Any such configuration contains at most $2d$ points of $\Gamma$, by taking $d$ lines.) Nonetheless, $\Gamma$ satisfies $\CB(r)$ by B\'{e}zout.
\end{example}

This example demonstrates the following. 

\begin{corollary} \label{cor:bound-is-tight}
The bound $|\Gamma| \leq (d+1)r+1$ in Conjecture \ref{conj:main} is tight.
\end{corollary}

We next see that, under the hypotheses of the conjecture, the points need not lie on a degree-$d$ rational curve. This is in contrast to the case $d=1$ (where the points will lie on a line) and the result of \cite{SU} (where $d$ is effectively $1.5$ and the points lie on a degree two, possibly reducible, rational curve).

\begin{example}[Elliptic curve] \label{ex:elliptic-curve}
Let $r=2$ and $d=3$. In $\mb{P}^3$, let $\Gamma$ consist of $9 = (3+1)2+1$ general points on a quartic elliptic curve $E$. Then no $8$ points of $\Gamma$ are a complete intersection of $E$ with a quadric surface (such complete intersections form only a $7$-dimensional family), so $\Gamma$ is $\CB(2)$. The unique $3$-dimensional plane configuration containing $\Gamma$ is the $\mb{P}^3$. Moreover, $\Gamma$ does not lie on any union of rational curves of total degree $3$. (See Question \ref{q:curves} for more discussion about points on possibly-reducible curves.)
\end{example}

Plane configurations of length greater than one are unavoidable:

\begin{example}[Two 2-planes]
Let $d=4$ and $r=3$. Let $\mc{P}$ be a split configuration of two 2-planes $P_1, P_2$. Let $C_i$ be a smooth conic on $P_i$ ($i=1, 2$) and let $\Gamma$ consist of $8$ general points from each $C_i$. Then $\Gamma$ is $\CB(3)$ and has $(d+1)r+1 = 16$ points. Moreover, $\mc{P}$ is the unique plane configuration of dimension $4$ containing $\Gamma$.
\end{example}

\begin{example}[$d$ general lines] \label{ex:d-lines}
Suppose $r \geq 2d-1$. Let $\mc{P}$ consist of $d$ general lines in $\mb{P}^{n}$, where $n > d > \frac{1}{2}n$. Let $\Gamma$ consist of $r+2$ general points from each line. Then $\Gamma$ is $\CB(r)$ by B\'{e}zout and $|\Gamma| = d(r+2) \leq (d+1)r+1$. The only dimension-$d$ plane configuration containing $\Gamma$ is the $d$ lines.
\end{example}

Notably, this example shows that plane configurations of arbitrary length arise in Conjecture \ref{conj:main} for sufficiently large $r$. See also Corollary \ref{cor:d-lines-bound}, which shows that skew configurations of maximum length $d$ arise only outside the range $1 \leq \frac{r}{d-1} \leq 2$.

\section{Key lemmas} \label{sec:key-lemmas}

We have seen that complements of $\Gamma$ by hyperplanes must satisfy lower-degree Cayley--Bacharach conditions. These conditions seem too weak to allow inductive approaches to Conjecture \ref{conj:main}. Our core approach is to establish that, under certain conditions, subsets of $\Gamma$ again satisfy $\CB(r)$, with the same value of $r$.

\subsection{Split configurations}

\begin{lemma} \label{lem:disjoint-planes}
Let $\Gamma \subset \mathbb{P}^n$ be a finite set of points. Let $P_1, P_2$ be disjoint, positive-dimensional planes, and suppose $\Gamma \subset P_1 \cup P_2$.

Let $r \geq 1$. Then $\Gamma$ is $\CB(r)$ if and only if $\Gamma \cap P_1$ and $\Gamma \cap P_2$ are both $\CB(r)$.
\end{lemma}
\begin{proof}
Let $p_i = \dim P_i$. First, we choose coordinates $[X_0 : \cdots : X_{p_1} : Y_0 : \cdots : Y_{p_2}]$ for $\mathrm{span}(P_1, P_2) \subseteq \mathbb{P}^n$ so that $P_1 = [X_0 : \cdots : X_{p_1} : 0 : \cdots : 0]$ and $P_2 = [0 : \cdots : 0 : Y_0 : \cdots Y_{p_2}]$. Write $\Gamma = \Gamma_1 \cup \Gamma_2$ where $\Gamma_i = \Gamma \cap P_i$.

$(\Rightarrow)$ Suppose $\Gamma$ satisfies $\CB(r)$. Let $F$ be a polynomial of degree $r$ vanishing at all but possibly one point of $\Gamma_1$. We claim that $F$ vanishes at the last point. Let $G(X, Y) = F(X, 0)$, that is, the polynomial obtained by deleting all monomials involving $Y_j$'s. We may assume $G$ is not identically zero. Now $G|_{P_1} = F|_{P_1}$, and $G$ vanishes identically on $P_2$, hence on all of $\Gamma_2$. By $\CB(r)$ for $\Gamma$, $G$ vanishes at the last point of $\Gamma_1$, so $F$ does also. \medskip

$(\Leftarrow)$ Suppose $\Gamma_1$ and $\Gamma_2$ satisfy $\CB(r)$. Let $F$ be a degree-$r$ polynomial vanishing at all but possibly one point of $\Gamma$, say $p \in \Gamma_1$. We may assume $F|_{P_1}$ is not identically zero, so by $\CB(r)$ for $\Gamma_1$, $F$ vanishes at the last point.
\end{proof}

We will primarily use this for pairs of disjoint planes, but the same argument applies to split configurations:

\begin{proposition}
Let $\mc{P} = \bigcup P_i$ be a \emph{split} plane configuration and let $\Gamma \subseteq \mc{P}$ be a finite set. Let $r \geq 1$. Then $\Gamma$ is $\CB(r)$ if and only if, for each $i$, $\Gamma \cap P_i$ is $\CB(r)$.
\end{proposition}

\begin{proposition} \label{prop:disjoint planes case}
Suppose Conjecture \ref{conj:main} holds up to $d-1$ for fixed $r$. Let $|\Gamma| \leq (d+1)r+1$. Let $A, B \subseteq \mathbb{P}^n$ be disjoint, positive-dimensional planes. Suppose $\Gamma \subset A \cup B$ and $\Gamma \cap A$ and $\Gamma \cap B$ are nonempty. Then $\Gamma$ lies on a plane configuration of dimension $d$.
\end{proposition}

\begin{proof}
Let $\Gamma_A = \Gamma \cap A$ and $\Gamma_B = \Gamma \cap B$. By counting, either 
\[|\Gamma_A| \leq dr+1 \text{ or } |\Gamma_B| \leq dr+1.\]
Without loss of generality, assume $|\Gamma_A| \leq dr+1$. Let $d_A$ be minimal such that $\Gamma_A$ lies on a plane configuration of dimension $d_A$. By Proposition \ref{lem:disjoint-planes}, $\Gamma_A$ is $\CB(r)$, so the earlier cases of the Conjecture imply $1 \leq d_A \leq d-1$. The inductive lower bound (Proposition \ref{prop:inductive-lower-bound}) implies $|\Gamma_A| \geq d_Ar+2$ by minimality of $d_A$, and so
\[|\Gamma_B| = |\Gamma| - |\Gamma_A| \leq (d-d_A+1)r - 1 < (d-d_A+1)r+1,\]
so $\Gamma_B$ lies on a plane configuration of dimension $d-d_A$. Combining these two plane configurations gives a plane configuration of dimension $d$ as desired.
\end{proof}

\subsection{Unions of two planes meeting at a point}
Let $A, B \subset \mathbb{P}^n$ be linear spaces meeting at a point $p$. Without loss of generality, assume $A$ and $B$ span $\mathbb{P}^n$. We may choose coordinates of the form $[X : \cdots Y_i \cdots : \cdots Z_j \cdots],$ where $p = [1 : 0 :\cdots : 0]$, $A = [X : \cdots Y_i \cdots : 0 \cdots 0]$ and $B = [X : 0 \cdots 0 : \cdots Z_j \cdots]$.
\begin{lemma} \label{lem:point-extension}
Let $F_A, F_B$ be polynomials of degree $r$ on $A$ and $B$, such that $F_A(p) = F_B(p)$. Then there exists a polynomial $F$ on $\mathbb{P}^n$ such that $F|_A = F_A$ and $F|_B = F_B$.
\end{lemma}
\begin{proof}
Restricting to $p$ means setting all $Y_i$ and $Z_j$ to $0$, so $F_A(p) = F_B(p)$ means the coefficient $c$ of $X^r$ is the same for both. Then put
\[F(X, Y, Z) = c X^r + (F_A(X, Y, 0) - c X^r) + (F_B(X, 0, Z) - c X^r),\]
where by $Y$ and $Z$ we mean the corresponding tuples of variables $Y_i$ and $Z_j$. Setting all $Z_j = 0$ recovers $F_A$ and setting all $Y_i = 0$ recovers $F_B$.
\end{proof}

Now let $\Gamma \subset A \cup B$ be a finite set. Let $\Gamma_A = (\Gamma \cap A) \setminus \{p\}$ and $\Gamma_B = (\Gamma \cap B) \setminus \{p\}$.
\begin{lemma}
\label{lem:union-of-two-at-point}
Suppose $\Gamma$ satisfies $\CB(r)$.
\begin{itemize}
    \item[1.] At least one of $\Gamma_A$ and $\Gamma_A \cup \{p\}$ is $\CB(r)$.
    \item[2.] If $p \notin \Gamma$, at least one of $\Gamma_A$ and $\Gamma_B \cup \{p\}$ is $\CB(r)$.
    \item[3.] If $p \in \Gamma$, at least one of $\Gamma_A \cup \{p\}$ and $\Gamma_B \cup \{p\}$ is $\CB(r)$.
\end{itemize}
Note that we may exchange $A$ and $B$ in statements 1 and 2.
\end{lemma}
Statement (1) is the most important as it guarantees information about the points on $A$ (and, analogously, the points on $B$). We can visualize the other relationships according to a graph:

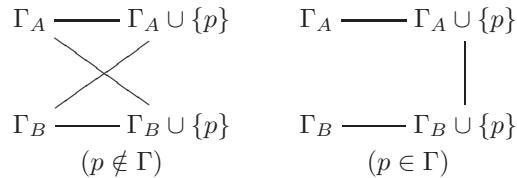
\begin{center}
\begin{tabular}{ccc}
\xymatrix{
\Gamma_A \ar@{-}[r] \ar@{-}[dr] & \Gamma_A \cup \{p\} \ar@{-}[dl] \\
\Gamma_B \ar@{-}[r] & \Gamma_B \cup \{p\} \\
}     &&
\xymatrix{
\Gamma_A \ar@{-}[r] & \Gamma_A \cup \{p\} \ar@{-}[d] \\
\Gamma_B \ar@{-}[r] & \Gamma_B \cup \{p\} \\
}
\\
$(p \notin \Gamma)$     && $(p \in \Gamma)$
\end{tabular}
\end{center}
Each edge indicates ``at least one of these two sets must be $\CB(r)$''.

\begin{proof}
We repeatedly use the following observation: 
\begin{itemize}
\item[$(*)$] Suppose $F_A$ is a polynomial of degree $r$ vanishing at $p$ and at all but one point of $\Gamma_A$. Then extending by $F_B = 0$ via Lemma \ref{lem:point-extension} contradicts $\Gamma$ being $\CB(r)$. Thus no such $F_A$ exists.
\end{itemize}
Effectively, $(*)$ says $\Gamma_A \cup \{p\}$ is ``almost $\CB(r)$'': the only degree $r$ polynomials that can leave out a point of $\Gamma_A \cup \{p\}$ are those leaving out $p$ (and therefore vanishing at all of $\Gamma_A$). \medskip

1. Suppose $\Gamma_A$ is not $\CB(r)$. Then there exists $F_A$ leaving out one point $a \in \Gamma_A$. By observation $(*)$, $F_A(p) \ne 0$. Suppose further that $\Gamma_A \cup \{p\}$ is not $\CB(r)$. Then there exists $\tilde F_A$ leaving out one point of $\Gamma_A \cup \{p\}$. Again, $(*)$ implies $\tilde F_A(p) \ne 0$, and so $\tilde F_A$ vanishes on all of $\Gamma_A$. But then a suitable linear combination of $F_A$ and $\tilde F_A$ vanishes on $\Gamma_A \setminus \{a\}$ (since both $F_A$ and $\tilde F_A$ do) and at $\{p\}$ (by choice of combination) and does not vanish at $\{a\}$ (since $F_A$ does not and $\tilde F_A$ does). This contradicts $(*)$. \medskip

2. Suppose $\Gamma_A$ is not $\CB(r)$ and let $F_A$ leave out $a \in \Gamma_A$. By observation $(*)$, $F_A(p) \ne 0$. Suppose $\Gamma_B \cup \{p\}$ is also not $\CB(r)$ and let $F_B$ be a polynomial vanishing at all but one point of $\Gamma_B \cup \{p\}$. By $(*)$ applied with $A$ and $B$ exchanged, $F_B$ must also leave out $p$ and vanish on all of $\Gamma_B$. By Lemma \ref{lem:point-extension}, there exists a global polynomial $F$ on $\mathbb{P}^n$ extending $F_A$ and $F_B$ (up to scalar). Since $p \notin \Gamma$, this contradicts $\Gamma$ being $\CB(r)$. \medskip

3. Suppose there exist $F_A$ and $F_B$ leaving out points of $\Gamma_A \cup \{p\}$ and $\Gamma_B \cup \{p\}$, respectively. By $(*)$, the left-out point must be $p$ in both cases. By Lemma \ref{lem:point-extension}, $F_A$ and $F_B$ extend to a global polynomial on $\mathbb{P}^n$ vanishing on $\Gamma_A \cup \Gamma_B$. Since $p \in \Gamma$, this contradicts $\Gamma$ being $\CB(r)$.
\end{proof}

\begin{proposition}\label{planes meeting at a point}
Suppose the Conjecture holds up to $d-1$ for fixed $r$. Let $|\Gamma| \leq (d+1)r+1$. Let $A, B \subseteq \mb{P}^n$ be positive-dimensional linear spaces meeting at a point $p$. Suppose $\Gamma \subset A \cup B$ and $(\Gamma \cap A) \setminus \{p\}$ and $(\Gamma \cap B) \setminus \{p\}$ are nonempty. Then $\Gamma$ lies on a plane configuration of dimension $d$.
\end{proposition}

\begin{proof}
Let $p, \Gamma_A, \Gamma_B$ be as above. By counting, either $|\Gamma_A| \leq dr$ or $|\Gamma_B| \leq dr$. Without loss of generality, assume the first inequality holds.

Let $S$ be whichever of the sets $\Gamma_A$ or $\Gamma_A \cup \{p\}$ is $\CB(r)$, according to Lemma \ref{lem:union-of-two-at-point}(1). Let $d_1$ be minimal such that $S$ lies on a plane configuration of dimension $d_1$. Then $|S| \leq dr+1$, so the earlier cases of the conjecture imply $1 \leq d_1 \leq d-1$.

By minimality of $d_1$, it follows $|S| \geq d_1r + 2$. Then since $|\Gamma_B| + |S| \leq |\Gamma|$,
\[|\Gamma_B| + 1 \leq |\Gamma| - |S| + 1 \leq (d-d_1+1)r,\]
so whichever of the sets $\Gamma_B$ or $\Gamma_B \cup \{p\}$ satisfies $\CB(r)$ lies on a plane configuration of dimension $d-d_1$.

So we are done if the two plane configurations we have constructed cover $\Gamma$. The remaining case is where $p \in \Gamma$ but is not included in either configuration. That is, $\Gamma_A$ and $\Gamma_B$ are both $\CB(r)$, but neither $\Gamma_A \cup \{p\}$ nor $\Gamma_B \cup \{p\}$ are $\CB(r)$. This contradicts Lemma \ref{lem:union-of-two-at-point}(3).
\end{proof}

\subsection{Bounding the length of the configuration}

The main use of the following proposition is to prove Corollary \ref{cor:d-lines-bound}, which gives conditions on $r$ and $d$ under which $\Gamma$ lies on a plane configuration of maximal length, i.e. $d$ disjoint lines.

\begin{proposition} \label{prop:balancing}
Suppose $\Gamma$ is $\CB(r)$ and is contained in a skew plane configuration $\mathcal{P} = \bigcup P_i$, with $\Gamma \cap P_i$ nonempty for all $i$. Then one of the following holds:
\begin{itemize}
\item[(i)] Each plane $P_i$ contains at least $\max(\ell(\mc{P}), r+2)$ points of $\Gamma$, or
\item[(ii)] Some plane $P_i$ contains fewer than $\ell(\mathcal{P})$ points, and also $\ell(\mathcal{P}) \geq r+2$.
\end{itemize}
\end{proposition}

\begin{proof}
Let $i$ be such that $\gamma_i = \#(\Gamma \cap P_i)$ is of minimal cardinality. Fix $x \in \Gamma \cap P_i$. We construct a collection of hyperplanes covering $\Gamma \setminus \{x\}$ as follows.

Observe that, given $j \ne i$ and a point $x' \in \Gamma \cap P_i \setminus \{x\}$, the plane $P' = \mathrm{span}(P_j, x')$ intersects $P_i$ only in $x'$, by the skewness assumption. We may extend it to a hyperplane $H'$ containing no additional points of $\Gamma \cap P_i$.

We first choose $\min(\gamma_i - 1, \ell(\mathcal{P}) - 1)$  hyperplanes $H'$ according to this observation, taking a different $j$ and a different $x' \in \Gamma \cap P_i \setminus \{x\}$ each time.

Suppose $\gamma_i \geq \ell(P)$. Then for each remaining point of $\Gamma \cap P_i \setminus \{x\}$, we take a hyperplane $H''$ containing that point and no other points of $\Gamma$, giving $\gamma_i - 1$ hyperplanes in all. The union of all these hyperplanes misses exactly the point $x \in \Gamma$. Since $\Gamma$ is $\CB(r)$, we see that $r < \gamma_i - 1$. By minimality of $\gamma_i$, we conclude that each plane in $\mc{P}$ contains at least $\max(r+2, \ell(\mc{P}))$ points. This gives case (i).

Suppose instead $\gamma_i < \ell(P)$. Then we instead extend the remaining $P_j$'s to hyperplanes $H''$ not containing $x$, giving a total of $\ell(\mc{P})-1$ hyperplanes covering $\Gamma \setminus \{x\}$. Then $r < \ell(\mc{P}) - 1$. This gives case (ii).
\end{proof}

\begin{corollary} \label{cor:d-lines-bound}
Let $\Gamma \subseteq \mb{P}^n$ be $\CB(r)$ and $|\Gamma| \leq (d+1)r+1$. Suppose the Conjecture holds, so $\Gamma$ lies on a plane configuration $\mc{P} = \bigcup P_i$ of dimension $d$, with $\Gamma \cap P_i$ nonempty for all $i$. By summing any intersecting planes of $\mc{P}$, assume $\mc{P}$ is skew.

Suppose $1 \leq \frac{r}{d-1} \leq 2$. Then the length of $\mc{P}$ is at most $d-1$.
\end{corollary}
\begin{proof}
Suppose instead that $\ell(\mc{P}) = d$, that is, $\Gamma$ lies on $d$ skew lines.

We are not in case (ii) of Proposition \ref{prop:balancing} since $d < r+2$. But then we must be in case (i), so each line contains at least $r+2$ points of $\Gamma$. This gives $|\Gamma| \geq d(r+2)$. Combining with $|\Gamma| \leq (d+1)r+1$, we see that $r \geq 2d-1$, a contradiction.
\end{proof}

\section{Proof of main theorem} \label{sec:proof-main}

In this section, we prove Theorem \ref{thm:main}. Throughout, $\Gamma \subseteq \mb{P}^n$ denotes a finite set satisfying $\CB(r)$, and such that $|\Gamma| \leq (d+1)r+1.$ We wish to show that $\Gamma$ lies on a plane configuration $\mc{P}$ of dimension $d$ and with length $\ell(\mc{P})$ bounded as in the statement of the theorem. Note that the case $r=0$ is obvious and the case $d=0$ is precisely Proposition \ref{prop:min-count}.

\begin{remark}[Setup] \label{rmk:setup}
Most of the proofs below begin with the following setup and notation. We let $\alpha \geq d+1$ be the maximum number of points of $\Gamma$ lying on a $d$-plane. Let $A$ be such a $d$-plane and let $\Gamma_A = \Gamma \cap A$. We denote the complementary set of points as $\Gamma_B := \Gamma \setminus \Gamma_A$, we let $\beta = |\Gamma_B|$ and we let $B$ be the linear span of $\Gamma_B$. We may assume $\beta > 0$ since otherwise $\Gamma$ lies on $A$. \qed
\end{remark}

\subsection{Fixed $r$, varying $d$}

The case $r=1$ holds essentially by definition:

\begin{proposition}\label{r=1}
If $\Gamma$ is $\CB(1)$ and $|\Gamma| \leq d+2$, then $\Gamma$ lies on a $d$-plane.
\end{proposition}

\begin{proof}
If the first $d+1$ points don't span a $d$-plane, add the last point and we're done. If they \emph{do} span a $d$-plane, that plane contains the last point by $\CB(1)$.
\end{proof}

\begin{remark}For $r=1$ and $d > 1$, the weaker conclusion ``$\Gamma$ lies on a $d$-dimensional plane configuration (of arbitrary length)'' is vacuous and does not require $\CB(1)$ at all: any $d+2$ points trivially lie on a $(d{-}1)$-plane and a line -- indeed they lie on $\lceil \tfrac{1}{2}(d+2) \rceil$ lines.
\end{remark}

We next consider the case $r=2$. This is the final case in which the plane configuration always has length $1$.

\begin{proposition}\label{r=2}
If $\Gamma$ is $\CB(2)$ and $|\Gamma| \leq 2d+3$, then $\Gamma$ lies on a $d$-plane.
\end{proposition}

\begin{proof}
Let $\alpha, \Gamma_A, A, \beta, \Gamma_B, B$ be as in Remark \ref{rmk:setup}. For contradiction, assume $\beta > 0$.

We reduce to the case where any $d+1$ points $\Gamma' \subseteq \Gamma_A$ lie on a $(d-1)$-plane. To see this, consider such a subset $\Gamma'$ and extend $B$ by the other $\alpha-d-1$ points of $\Gamma_A$. By excision $\Gamma_B$ is $\CB(1)$, so $\dim(B) \leq \beta-2$, so the resulting plane $P$ has dimension at most
\[(\beta-2) + (\alpha-d-1) \leq d.\]
If $\Gamma' \subset P$, then $\Gamma \subset P$ and the proof is complete. Otherwise, $\Gamma \setminus P = \Gamma' \setminus P$ is $\CB(1)$ by excision and so lies on a $(|\Gamma'\setminus P|-2)$-plane. Returning the points of $\Gamma' \cap P$ to the set, we see $\Gamma'$ spans a plane of dimension at most $|\Gamma'|-2 = d-1$, as desired.

Now, by maximality, $\Gamma_A$ must span the $d$-plane $A$. But then $\Gamma_A$ contains $d+1$ linearly independent points. This is a contradiction, so $\beta = 0$.
\end{proof}

Notice that as soon as $r\geq 3,$ we can no longer conclude that $\Gamma$ lies on a plane configuration of length one, as we see in the following example.

\begin{example}
Let $\Gamma$ be a set of $10 = (2+1)3+1$  points on two skew lines, five points on each line. Then $\Gamma$ satisfies $\CB(3)$ by B\'{e}zout, but the only plane configuration of dimension two that $\Gamma$ lies on has length two.
\end{example}

See Example \ref{ex:d-lines} for a generalization to $d$ skew lines, for sufficiently large $r$.

\subsection{Fixed $d$, varying $r$}

We prove the cases $d=1, 2, 3$. The case $d=1$ is known. We give a short proof using our approach.

\begin{theorem}[Case $d=1$; see \cite{BCD}]\label{thm:d=1}
If $\Gamma$ is $\CB(r)$ and $|\Gamma| \leq 2r+1$, then $\Gamma$ lies on a line.
\end{theorem}

\begin{proof}
Let $\alpha, \Gamma_A, A$ and $\beta, \Gamma_B, B$ be as in Remark \ref{rmk:setup}. Assume for contradiction that $\beta > 0$. By excision, $\Gamma_B$ is $\CB(r-1)$ and, by counting, $\beta \leq 2(r-1)+1$ (since  $\alpha \geq 2$). By induction on $r$, $\Gamma_B$ lies on a line.

By Proposition \ref{prop:min-count}, we have $(r-1)+2 \leq \beta$. By maximality of $\alpha$ we have $\beta \leq \alpha$, which gives $|\Gamma| \geq 2r+2$, a contradiction.
\end{proof}

\begin{theorem}[Case $d=2$]\label{thm:d=2}
If $\Gamma$ is $\CB(r)$ and $|\Gamma| \leq 3r+1$, then $\Gamma$ lies on a 2-plane or two skew lines.
\end{theorem}

\begin{proof}
We have established the cases $r=1$ and $r=2$, so we assume $r \geq 3$ and use induction on $r$.

Let $\alpha, \Gamma_A, A$ and $\beta, \Gamma_B, B$ be as in Remark \ref{rmk:setup}. We may assume that $\beta > 0$. We have $\alpha \geq d+1 = 3$, and so
\[\beta = |\Gamma| - \alpha \leq (3r+1) - 3 = 3(r-1)+1.\]
By excision, $\Gamma_B$ is $\CB(r-1)$, so by induction, $\Gamma_B$ lies on a $2$-plane or on two skew lines. In fact, we will show that it lies on a line. We consider both cases.

\vspace{.2cm}
\textit{Case 1}: $\Gamma_B$ lies on a $2$-plane. By maximality of $\alpha$ and the fact $r \geq 3$, we have 
\[\beta \leq \tfrac{1}{2}|\Gamma| = \tfrac{3}{2}r + \tfrac{1}{2} \leq 2(r-1)+1.\]
Since $\Gamma_B$ is $\CB(r-1)$, it lies on a line by Theorem \ref{thm:d=1}.

\vspace{0.2cm}
\textit{Case 2}: $\Gamma_B$ lies on two skew lines. Suppose each line contains at least one point of $\Gamma_B$. Since $\Gamma_B = \Gamma \setminus A$, $\Gamma_B$ is $\CB(r-1)$ by excision. Since the two lines are skew, the points of $\Gamma_B$ on each individual line are $\CB(r-1)$ by Lemma \ref{lem:disjoint-planes}. Therefore each line contains at least $r+1$ points by Lemma \ref{prop:min-count}. One line together with one point from the other line span a $2$-plane, so by maximality of $\alpha$, $\alpha \geq r+2$. Adding these up, we obtain $\alpha + \beta \geq 3r+4$, a contradiction.

We now finish the proof. If $A$ and $B$ are disjoint, then $\Gamma$ lies on plane configuration of dimension $2$ by Proposition \ref{prop:disjoint planes case}. If $A$ and $B$ meet a point, the same conclusion follows from Proposition \ref{planes meeting at a point}. Otherwise, $A \supseteq B$ and $\Gamma$ lies on $A$ itself. \qedhere

\end{proof}

\begin{theorem}[Case $d=3$]\label{d=3}
If $\Gamma$ is $\CB(r)$ and $|\Gamma| \leq 4r+1$, then $\Gamma$ lies on a plane configuration $\mathcal{P}$ of dimension $3$.
\end{theorem}

Note that if $r \leq 4$, then we may further conclude $\ell(\mc{P}) \leq 2$ by Corollary \ref{cor:d-lines-bound}.

\begin{proof}
The statement holds for $r=1$ and 2 by Propositions \ref{r=1} and \ref{r=2}. We can thus assume $r \geq 3$ and show the statement holds in general by induction on $r$.

Let $\alpha, \Gamma_A, A$ and $\beta, \Gamma_B, B$ be as in Remark \ref{rmk:setup}. We may assume $\beta > 0$. We have $\Gamma_B = \Gamma \setminus A$, so by excision $\Gamma_B$ is $\CB(r-1)$. So, by the basic lower bound (Proposition \ref{prop:min-count}), $\beta \geq r+1$. On the other hand,
\[\beta \leq |\Gamma| - \alpha \leq 4(r-1)+1\]
so by induction, $\Gamma_B$ lies on a plane configuration of dimension $3$. We consider the possibilities below; in fact we will reduce to the case where $\Gamma_B$ lies on a single line.

\vspace{.2cm}
\textit{Case 1:} $\Gamma_B$ lies on 3 disjoint lines. If $\Gamma_B$ lies on just two lines, then it lies in a 3-plane, which is covered by Case 3. Thus, assume it lies in three disjoint lines but not in a 3-plane. We obtain a contradiction as follows.

Choose two of the lines whose union contains at least $\tfrac{2}{3}\beta$ points of $\Gamma_B$. The span of those lines is a 3-plane, so by maximality of $\alpha$, $\tfrac{2}{3}\beta \leq \alpha$. We calculate 
\[\tfrac{2}{3}\beta + \beta \leq \alpha + \beta \leq 4r+1\]
and so $\beta \leq (12r+3)/5.$ If $r \geq 5,$ this gives $$\beta \leq 3(r-1)+ 1,$$ so by Theorem \ref{thm:d=2}, $\Gamma_B$ lies on a dimension 2 plane configuration, which means it must lie on a 3-plane, contradicting our assumption.

If $r=3$, then by the above inequality, $$\beta \leq (12r+3)/5 = 39/5.$$ So $\beta \leq 7,$ which means at least one of the three lines contains $\leq 2$ points of $\Gamma_B.$ The other two lines span a 3-plane, which means the leftover $\leq 2$ points are $\CB(1)$, a contradiction. An analogous argument shows that if $r=4$, one of the lines has $\leq 3$ points, and whichever of those are not in the span of the other two lines must be $\CB(2)$, a contradiction. This finishes Case 1.

\vspace{.2cm}
\textit{Case 2:} $\Gamma_B$ lies on the union of a line and a 2-plane. Again, we assume $\Gamma_B$ does not lie on a $3$-plane. So, the plane and line must be disjoint.

Therefore, the subsets of $\Gamma_B$ on the plane and the line each satisfy $\CB(r-1)$ by Lemma \ref{lem:disjoint-planes}. Thus there are $\geq 2(r-1)+2 =2r$ points of $\Gamma_B$ on the 2-plane and $\geq r-1+2 =r+1$ on the line, which gives $\beta \geq 3r+1$. The span of one point of $\Gamma_B$ on the line along with the 2-plane is a 3-plane, which by hypothesis implies that $\alpha \geq 2r+1,$ and in turn $|\Gamma| \geq 5r+2,$ a contradiction. 

This leaves one final case.

\vspace{.2cm}
\textit{Case 3:} $\Gamma_B$ lies on a 3-plane.  By maximality of $\alpha$, we have $\beta \leq \alpha$, so since $\beta + \alpha = |\Gamma| = 4r+1$, we obtain $\beta \leq 2r + \tfrac{1}{2}$, hence $\beta \leq 2r$.  We split this up into two cases.

\vspace{.2cm}
\textit{Case 3a:} $\beta = 2r.$ 

Then $\alpha = 2r$ or $\alpha = 2r+1$. Since $r>2,$ we have $2r <  3r-2 =3(r-1)+1,$ which by Proposition \ref{thm:d=2} implies that the $\beta = 2r$ points of $\Gamma_B$ lie on a 2-plane or two skew lines. If two lines, the set of points of $\Gamma_B$ on each line are $\CB(r-1)$, so each line must contain at least $(r-1)+2$ points of $\Gamma_B$, a contradiction. Thus, in this case, $\Gamma_B$ lies on a 2-plane $B$. But then for any point $a \in \Gamma_A$, the $3$-plane $\mathrm{span}(B, \{a\})$ contains exactly the $2r+1$ points $\Gamma_B \cup \{a\}$, otherwise contradicting the maximality of $\alpha$. Thus, by excision,
\[\Gamma \setminus \mathrm{span}(B, \{a\}) = \Gamma_A \setminus \{a\}\]
is $\CB(r-1)$ and has exactly $2r$ points. By the same argument as used for $B$ above, $\Gamma_A \setminus \{a\}$ must lie on a 2-plane. Since $a$ was chosen arbitrarily, and $\alpha > 4$, all of $\Gamma_A$ must lie on a 2-plane, a contradiction, which completes this case.

\vspace{.2cm}
\textit{Case 3b:} $\beta \leq 2r-1.$ 

Since $2r-1 = 2(r-1)+1$, we deduce that $\Gamma_B$ lies on a line by Theorem \ref{thm:d=1}, that is, $B = \mathrm{span}(\Gamma_B)$ is a line.

The desired statement now follows by Proposition \ref{prop:disjoint planes case} (if $A$ and $B$ are disjoint), by Proposition \ref{planes meeting at a point} (if $A$ and $B$ meet at a point) or directly (if $A \supset B$). \qedhere

\end{proof}

\subsection{The case $d=4, r=3$}

Having established Conjecture \ref{conj:main} for $d \leq 3$ and for $r \leq 2$, we consider the next case not covered, namely $d=4$, $r=3$. Our argument uses Propositions \ref{planes meeting at a point} and \ref{prop:disjoint planes case} to handle nearly all cases. We give an ad hoc argument for the remaining case, in which $\Gamma$ lies on two hyperplanes meeting along a line.

\begin{theorem}[Case $d=4$, $r=3$]
Let $d=4$, let $\Gamma$ be $\CB(3)$ and suppose $|\Gamma| \leq (d+1)r+1 = 16$. Then $\Gamma$ lies on a plane configuration of dimension $4$.
\end{theorem}

\begin{proof}
We may assume $|\Gamma| \geq 14$; otherwise we are done by the case $r=3, d=3$.

Let $\alpha, \Gamma_A, A$ and $\beta, \Gamma_B, B$ be as in Remark \ref{rmk:setup}. We assume $\beta > 0$, so in particular $\Gamma_B$ is $\CB(2)$, $\beta \geq 4$ and $\alpha \leq 12$.

Most cases are handled as follows. First, if $\dim B = 1$, we are done by Proposition \ref{prop:disjoint planes case} if $A \cap B$ is empty and by Proposition \ref{planes meeting at a point} if $A \cap B$ is a point. Since $\Gamma_B$ is $\CB(2)$, we may in particular assume $\beta \geq 6$, and we also have $\dim(B) \leq \lceil \tfrac{\beta-3}{2} \rceil$. Second, we may extend $B$ to a $\mathbb{P}^4$ by adding points. So, by maximality of $\alpha$, we get
\[\beta + 4 - \lceil \tfrac{\beta-3}{2} \rceil \leq \beta + (4 - \dim(B)) \leq \alpha.\]
By numerical consideration of $\alpha$ and $\beta$, this covers nearly every case. The remaining ones are below; we also assume $\dim(B) \geq 2$ and $\dim(A \cap B) \geq 1$.

{$\mathbf{ (\alpha, \beta) = (8, 6) \text{ or } (9, 7):}$} By $\CB(2)$, $\dim(B) = 2$. We must have $\Gamma \cap A \cap B$ empty or we contradict the maximality of $\alpha$ by extending $\Gamma_B$ by two points. Therefore we can write $\Gamma_A = \Gamma \setminus B$, so $\Gamma_A$ is $\CB(2)$ by excision. Then since $|\Gamma_A| \leq 9$, $\Gamma_A$ lies on a $\mathbb{P}^3$ by the case $r=2, d=3$ of the conjecture (Proposition 5.4). This contradicts the choice of $A$.

{$\mathbf{(\alpha, \beta) = (9, 6) \text{ or } (10, 6):}$} By $\CB(2)$ and the discussion above, $\dim B = 2$ and $A \cap B$ is a line $L$. By construction (as in Remark \ref{rmk:setup}), any points of $\Gamma$ lying on $L$ have been counted among the $\alpha$ points of $\Gamma_A$.

Suppose $\Gamma \cap L$ is nonempty. Then $|\Gamma \setminus B| \leq 9$, hence by $\CB(2)$ lies on a $3$-plane $A' \subset A$. Then $A'$ intersects $L$ in at least a point, and does not contain $L$ since $\Gamma_A$ does not lie on a $\mathbb{P}^3$. Thus $A'$ and $B$ cover $\Gamma$ and intersect in a point, so we are done by Lemma \ref{planes meeting at a point}.

Suppose $\Gamma \cap L$ is empty. We are done as above if $\alpha=9$. For $\alpha=10$, we consider all $2$-planes $A'$ such that $L \subset A' \subset A$ and $A'$ contains at least one point of $\Gamma_A$. Then the set $\Gamma \setminus \mathrm{span}(A', B)$ is $\CB(2)$, hence this set lies on some plane $C \cong \mathbb{P}^3$ (since $\alpha-1 \leq 9$). Since $\Gamma_A$ does not lie on a $\mathbb{P}^3$, we must have $C \not\supset \Gamma \cap A'$ (this containment would imply $C \supset \Gamma_A$). So, by $\CB(1)$ (excising $B$ and $C$), $\Gamma \cap A'$ contains at least 3 points. Since this is true for any choice of $A'$, combining $B$ with two such choices yields a $\mathbb{P}^4$ containing at least $\beta+3+3 > \alpha$ points, a contradiction.
\end{proof}

Finally, to check that the plane configuration $\mc{P}$ can be taken to have length $\ell(\mc{P}) \leq 2$: we have $\ell(\mc{P}) \leq 3$ by Corollary \ref{cor:d-lines-bound}. Similar considerations show that the case $\ell(\mc{P}) = 3$ (i.e., two lines and a $2$-plane) does not arise.

\section{Codimension two complete intersections} \label{sec:codim-2-ci}

In this section, we give proofs of Proposition \ref{codim2prop} and Theorem \ref{codim2thm}, which quickly follow from the main theorem and conjecture.

\begin{proof}[Proof of Proposition \ref{codim2prop}] 

Recall that $X \subset \P^{n+2}$ is a complete intersection of hypersurfaces $Y_a$ and $Y_b$ of degrees $a$ and $b$, respectively. 

For part (a), let $f: X \dashrightarrow \P^n$ be a dominant rational map, and let $\Gamma \subset X$ be a general fiber of $f$. The condition of ``lying on a plane configuration of dimension $b-1$'' is closed, so it is enough to show it for $\Gamma$ to conclude it for \emph{all} fibers. The result of Bastianelli \cite[Proposition 4.2]{Bas} mentioned in the introduction implies that $\Gamma$ satisfies the Cayley--Bacharach condition with respect to $K_X = (a + b - n -3)H$, where $H$ is a hyperplane section. That is, $\Gamma$ is $\CB(a + b - n -3)$. We claim:

$$|\Gamma| \leq a(b-1) \leq ((b-1) +1)(a+b-n-3) + 1 = (d+1)r+1.$$
The first inequality is by hypothesis for (a), which requires $\deg(f)$, and thus the number of points in a general fiber, to be $\leq a(b-1)$. The second inequality is implied by the assumption $a+1 \geq b(3+n-b)$. We then apply Conjecture \ref{conj:main} to conclude that $\Gamma$ lies on a plane configuration of dimension $b-1$, which completes the proof of part (a).

For part (b), we first note that the statement is trivial if $n+2 \leq b-1$ since we have $\Gamma \subseteq \mathbb{P}^{n+2}$, so we can assume that $n+2 \geq b.$ The statement is also trivial if $\dim(X) = n = 0,$ since $X$ is irreducible and thus a point. So additionally, assume $n+ 2 \geq 3.$

It suffices to show that the minimal degree of a dominant rational map $X \dashrightarrow \mathbb{P}^n$ (called the \emph{degree of irrationality of $X$} or $\mathrm{irr}(X)$) satisfies $\mathrm{irr}(X) \leq a(b-1)$. For this, we exhibit such a rational map of degree $\leq a(b-1)$.

Let $F(Y_b)$ be the Fano variety of lines on $X$. It is a well-known lemma that $\dim(F(Y_b)) \geq 2(n+2)-b-3.$ In particular, as long as $2(n+2)-b-3 \geq 0$, $F(Y_b)$ is non-empty. See, for example, the beginning of \cite[Section 4]{Beh} for an explanation of this fact. This inequality is guaranteed by our assumption $n+2 \geq b$.

So let $\ell \subset Y_b$ be a line, so projection from $\ell$ yields a dominant rational map 
\[f : X \dashrightarrow \P^n.\] Notice that either $\ell \cap X$ has length $a$ or else $\ell$ is contained in $X$. In the first case, $\deg(f) = a(b-1)$. In the second case, $\deg(f) = (a-1)(b-1)$ by the excess intersection formula. Either way, $\deg(f) \leq a(b - 1).$ This completes the proof.
\end{proof}

Theorem \ref{codim2thm} is a special case of Proposition \ref{codim2prop}.

\begin{proof}[Proof of Theorem \ref{codim2thm}]
Setting $b = 4$ and $Y_b = Y$ in Proposition \ref{codim2prop}, we obtain $d = 3$ and $r = a-n+1.$ By Theorem \ref{thm:main}, the conjecture holds in this case, and the inequality $a+1 \geq b(3+n-b)$ follows from the hypothesis $a \geq 4n-3$ (again by setting $b=4$). Thus, we can apply Proposition \ref{codim2prop}, which completes the proof.
\end{proof}

\section{Further considerations} \label{sec:conclusion}

As stated in the introduction, an immediate question raised by our Conjecture is to determine how the length of the plane configuration relates to $r$ and $d$.

\begin{question}
Let $\Gamma \subset \mathbb{P}^n$ be a finite set satisfying $\CB(r)$ and such that $|\Gamma| \leq (d+1)r+1$. Assuming Conjecture \ref{conj:main} holds, what can be said about the minimum length of a plane configuration $\mc{P}$ containing $\Gamma$?
\end{question}

Our main result shows that $\ell(\mc{P}) \leq 2$ for small values of $r$ and $d$; it's straightforward to exhibit examples in which $\ell(\mc{P}) \geq 3$ for higher values of $r$ and $d$. Corollary \ref{cor:d-lines-bound} shows, at the other extreme, that the case $\ell(\mc{P}) = d$ occurs only outside the range 
\[1 \leq \frac{r}{d-1} \leq 2.\]
It would be interesting to generalize this result to a bound on $\ell(\mc{P})$ dependent on $\tfrac{r}{d-1}$ or similar data.

\subsection{The Cayley--Bacharach locus of the Hilbert scheme}

Let $\mc{H} = \mathrm{Hilb}(\mb{P}^n, m)$ be the Hilbert scheme of $m$ points in $\mathbb{P}^n$. The Cayley--Bacharach condition gives a locally-closed subset $\mc{CB}(r) \subset \mc{H}$ corresponding to $m$-tuples of distinct points satisfying $\CB(r)$. 

\begin{problem}
Understand the geometry of $\mc{CB}(r)$ and its closure in $\mc{H}$, in particular its dimension, smooth locus, irreducible components and general points.
\end{problem}

We also observe that the conclusion of Conjecture \ref{conj:main}, of being contained in a plane configuration, is a \emph{closed} condition, hence applies to $\overline{\mc{CB}(r)}$ if it applies to $\mc{CB}(r)$. Both loci $\mc{CB}(r)$ and $\overline{\mc{CB}(r)}$ can be further decomposed into closed subsets according to the combinatorial type of the plane configuration,
\[
\mc{CB}(r) = \bigcup_{\ell, \vec{d}} \mc{CB}(r)_{\ell, \vec{d}},
\qquad
\overline{\mc{CB}(r)} = \bigcup_{\ell, \vec{d}} \overline{\mc{CB}(r)}_{\ell, \vec{d}}
\]
where $\ell$ denotes the length, $\vec{d} = (d_1, \ldots, d_\ell)$ the tuple of dimensions and the resulting loci parametrize those $\Gamma$ for which there exists a plane configuration $\mc{P}$ of the specified type, containing $\Gamma$.

\begin{problem}
Determine for which combinatorial types the corresponding subsets $\mc{CB}(r)_{\ell, \vec{d}}$ and $\overline{\mc{CB}(r)}_{\ell, \vec{d}}$ are nonempty, and determine their dimensions and geometry.
\end{problem}

There is also a natural analog of the Cayley--Bacharach condition for a subscheme $\xi \subseteq \mathbb{P}^n$ of finite length $m$: we may say $\xi$ is $\CB(r)$ if, whenever a homogeneous degree $r$ polynomial $F$ vanishes on a closed subscheme $\xi' \subseteq \xi$ of length $m-1$, then $F$ vanishes on $\xi$. This condition is again only \emph{locally} closed and gives another locus $\mc{CB}'(r)$ containing $\mc{CB}(r)$. 

\begin{question}
Which $\CB(r)$ subschemes can be smoothed to \emph{reduced} $\CB(r)$ subschemes? That is, what is $\mc{CB}'(r) \cap \overline{\mc{CB}(r)}$?
\end{question}

\subsection{Matroids}

Conjecture \ref{conj:main} and several of our arguments are highly suggestive of the theory of matroids and matroid varieties \cite{GGMS}. It would be interesting to understand this connection. Variants of the classical Cayley--Bacharach theorem have also recently arisen in studying ideals of matroid varieties via the Grassmann--Cayley algebra \cite{STW}. A basic question is as follows:

\begin{question}
Which matroids arise from sets of points $\Gamma \subseteq \mathbb{P}^n$ satisfying the assumptions of Conjecture \ref{conj:main}?
\end{question}

We may also approach the Cayley--Bacharach condition itself in purely matroidal terms. We say that a matroid $M$ satisfies the {\bf matroid Cayley--Bacharach} condition $\mCB(r)$ if, whenever a union of $r$ flats contains all but one point of $M$, the union contains the last point. (For a representable matroid, this is equivalent to restricting the geometric Cayley--Bacharach condition to fully reducible hypersurfaces.)

\begin{question}
Does the analog of Conjecture \ref{conj:main} hold for matroids $M$ satisfying $\mCB(r)$? That is, must $M$ be covered by a union of flats of specified dimensions? 
\end{question}

Some of the arguments of this paper are effectively entirely matroid-theoretic and give positive answers to this question in small cases. The arguments that are not matroidal (notably Lemma \ref{lem:union-of-two-at-point}) can sometimes be replaced, at least for low values of $d$ and $r$, by substantially longer matroidal arguments.

\subsection{Curves}

Finally, we revisit the question of describing when $\Gamma$ lies on a union of curves of specified degree and arithmetic genus. As shown in \cite{SU}, if $\Gamma$ is $\CB(r)$ and $|\Gamma| \leq \tfrac{5}{2}r+1$, then $\Gamma$ lies on a curve, possibly reducible, of degree at most $2$. 

These considerations are closely related to $\Gamma$ not imposing independent conditions on degree $r$ polynomials, a weaker condition than $\CB(r)$. For example, Castelnuovo \cite{Cas} gave conditions under which, if a finite set $\Gamma$ fails to impose independent conditions on quadrics, then it lies on a rational normal curve.

\begin{question} \label{q:curves}
Let $c > 0$ be a positive real number. Let $\Gamma \subseteq \mathbb{P}^n$ be a finite set of points satisfying $\CB(r)$ and such that $|\Gamma| \leq cr+1$. What is the minimum $d$ such that every such $\Gamma$ lies on a (possibly reducible) curve of degree $d$? What is the arithmetic genus of such a curve?
\end{question}

Thus the result in \cite{SU} is that for $c = \tfrac{5}{2}$, the minimum degree is $2$ and the curve has genus $0$. If it could be shown that for $c=d+1$, $\Gamma$ lies on a degree-$d$ curve, this would imply Conjecture \ref{conj:main} as such a curve lies on a plane configuration of dimension $d$. Unfortunately, this is false: for $r=d=2$, $7$ general points on a $2$-plane satisfy $\CB(2)$, but do not lie on a conic or a union of two lines. For $r=2$, $d=3$, Example \ref{ex:elliptic-curve} of nine general points on a quartic elliptic curve $E \subset \mathbb{P}^3$ is similar: $E$ itself, of degree $4$ and genus $1$, is a minimal degree curve containing $\Gamma$.

This question can be formulated more generally: it would be interesting to know when $\Gamma$ lies on special surfaces or varieties of higher dimension, on special plane configurations, and so on.

\begin{question}
With $c$ and $\Gamma$ as in Question \ref{q:curves}, what can be said about the geometry of $\Gamma$?
\end{question}

% \bibliographystyle{amsplain}
\bibliographystyle{amsalpha}
\bibliography{CBbib.bib}

\end{document}